%% file: DecompMain.tex
\documentclass[a4paper]{article}
\input{macros}

\input{environs}
\newcommand{\anan}{an-annular}
\author{Gareth Wilkes}
\title{Profinite completions, cohomology and JSJ decompositions of compact 3-manifolds}
\begin{document}
\maketitle

\input{intro}

\section{Preliminaries}
\input{IntroRelCoh}
\input{JSJbackgd}

\input{ProfProps}

\input{ProfAtorAnan}

\input{MalnPrComp}

\input{MainTheorems}

\bibliographystyle{alpha}
\bibliography{decomp}
\end{document}

%% file: macros.tex
\usepackage{amsmath}
\usepackage{amsfonts}
\usepackage{amssymb}
\usepackage{amsthm}
\usepackage{mathtools}
\usepackage{tikz}
\usepackage{tikz-cd}
\usepackage{enumerate}
\usepackage{layout}
\usepackage{mathabx}
\usepackage{bm}
\usetikzlibrary{patterns}

\DeclareMathOperator{\cd}{cd}

\newcommand{\bdy}{\ensuremath{\partial}}

\newcommand{\iso}{\ensuremath{\cong}}
\newcommand{\Z}[1][]{\ensuremath{\mathbb{Z}_{#1}}}

\newcommand{\nsgp}[1][]{\ensuremath{\triangleleft_{#1}}}

\newcommand{\gp}[1]{\ensuremath{\langle #1\rangle}}

\newcommand{\famS}[0]{\ensuremath{\mathcal{S}}}

\newcommand{\lqt}{\backslash}

%% file: environs.tex
\newtheorem{theorem}{Theorem}[section]

\newtheorem{prop}[theorem]{Proposition}

\newtheorem{lem}[theorem]{Lemma}
\newtheorem{clly}[theorem]{Corollary}

\theoremstyle{definition}
\newtheorem{defn}[theorem]{Definition}
\newtheorem*{cnv}{Conventions}

\theoremstyle{remark}
\newtheorem*{rmk}{Remark}

\theoremstyle{plain}

\newcounter{introthmcount}
\theoremstyle{definition}

%% file: intro.tex
\begin{abstract}
In this paper we extend previous results concerning the behaviour of JSJ decompositions of closed 3-manifolds with respect to the profinite completion to the case of compact 3-manifolds with boundary.

We also illustrate an alternative and perhaps more natural approach to part of the original theorem, using relative cohomology to analyse the actions of an-annular atoroidal groups on profinite trees.
\end{abstract}
\section*{Introduction}
Several recent papers have focused on those properties of 3-manifolds which can be detected via the finite quotients of the fundamental group---or equivalently via the profinite completion of the fundamental group. For example the geometry of a 3-manifold is determined by the profinite completion \cite{WZ14}, as is whether the manifold fibres over the circle \cite{Jaikin17}. There has also been much progress made towards answering the question of when two 3-manifold groups can have the same finite quotients \cite{BRW17, funar13, Wilkes16, Wilkes18}.

One of the most important tools in the study of irreducible 3-manifolds and their fundamental groups is the JSJ decomposition, a canonical graph of spaces decomposition of the 3-manifold with annuli and tori as edge spaces. The vertex spaces of such a decomposition are now known to always be geometric and therefore we have strong control over the fundamental group of a 3-manifold via its JSJ decomposition.

It is therefore interesting to study whether 3-manifolds whose fundamental groups have isomorphic profinite completions must have similar JSJ decompositions in the sense that the underlying graphs are isomorphic and corresponding vertex spaces have fundamental groups with isomorphic profinite completions. This was proved for irreducible 3-manifolds with toroidal boundary by Wilton and Zalesskii \cite[Theorem B]{WZ17}. The  principal aim of this paper is to extend this theorem to the case where the 3-manifold may have incompressible boundary of arbitrary genus. This is accomplished in Theorem \ref{JSJgeneral}.

We remark in passing that analogous theorems for prime decompositions of 3-manifolds are also known, by \cite[Theorem A]{WZ17} in the closed case and \cite[Theorem 6.22]{WilkesRelCoh} in the bounded case.

Wilton and Zalesskii remarked in \cite{WZ17} that one of the key parts of the proof of the JSJ decomposition (analysing the possible actions of the profinite completion of the fundamental group of a cusped hyperbolic 3-manifold on profinite trees) could `no doubt' be handled by developing a theory of relative cohomology of profinite groups. This theory has now been developed in \cite{WilkesRelCoh} and a secondary aim of this paper is to show that it does indeed have this application. This may be found in Theorem \ref{MainAnanThm}.

The use of relative cohomology is perhaps a more natural argument than the original analysis in \cite{WZ17}. In \cite{WZ17} cusped hyperbolic manifolds were handled by Dehn filling the cusps to obtain a closed hyperbolic manifold in such a way that the action on a profinite tree was not disturbed too much. The relative cohomology argument essentially concerns the absence of `tori' and `annuli' in the profinite completion and the obstruction this gives to splittings, which seems closer to the original spirit of a JSJ decomposition.

{\bf Acknowledgements} The author would like to thank his supervisor Marc Lackenby for reading this paper. The author was supported by the EPSRC and by Pembroke College, Oxford.
\begin{cnv} The following conventions will be in force through the paper.
\begin{itemize}
\item Generally profinite groups will be denoted with Roman letters $G$ or $H$ and discrete groups by Greek letters $\Gamma$, $\Lambda$ et cetera.
\item Maps of topological groups or modules should be assumed to be continuous homomorphisms in the appropriate sense.
\item All 3-manifolds will be compact, orientable and connected.
\item As we are interested in group-theoretic properties we will assume that 3-manifolds never have spherical boundary components. 
\item A finite graph of discrete groups will be denoted $(X,\Gamma_\bullet)$ where $X$ is a finite graph and $\Gamma_x$ for $x\in X$ denotes an edge or vertex group (with similar notation for graphs of profinite groups).
\end{itemize}
\end{cnv}

%% file: IntroRelCoh.tex
\subsection{Group pairs and relative cohomology}
Here we collect various definitions and properties of group pairs and relative cohomology. As in \cite{WilkesRelCoh} one should {\it in sensu stricto} only consider families of subgroups of a profinite group which are `continuously indexed by a profinite space'. However the issues associated to this will not arise in this paper so we shall ignore it to simplify the exposition.
\begin{defn}
A profinite group pair $(G,\famS)$ consists of a profinite group $G$ and a family \famS\ of closed subgroups $S_x$ of $G$ indexed over a set $X$---that is, a function $S_\bullet$ from $X$ to the set of closed subgroups of $G$. We allow repetitions in this family (that is, the function $S_\bullet$ need not be injective). 
\end{defn}
\begin{defn}
For a group $G$ and a collection $\famS=\{S_i\}_{i\in I}$ of subgroups of $G$ indexed by a set $I$, we define $\|\famS\| = |\{i\in I : S_i\neq 1\}|\in [0,\infty]$. 
\end{defn}
\begin{defn}
Let $G$ be a profinite group and let $\famS=\{S_x\}_{x\in X}$ be a family of subgroups of $G$ indexed by a set $X$. Let $H$ be a closed subgroup of $G$. Fix a section $\sigma\colon H\lqt G\to G$ of the quotient map $G\to H\lqt G$. Define the family of subgroups 
\[\famS^H_\sigma = \left\{ H\cap \sigma(y)S_x \sigma(y)^{-1} \mid x\in X, y\in H\lqt G/S_x\right\} \]
indexed over the set 
 \[H\lqt G/\famS=\bigsqcup_{x\in X} H\lqt G/S_x \]
\end{defn}
Changing the section $\sigma$ only affects the family $\famS^H_\sigma$ by changing its members by conjugacy in $H$, and we will henceforth ignore $\sigma$. 
\begin{defn}
Let $G$ and $H$ be profinite groups, let $\famS=\{S_x\}_{x\in X}$ be a family of subgroups of $G$ indexed by a set $X$ and let ${\cal T}=\{T_y\}_{y\in Y}$ be a family of subgroups of $H$ indexed by a set $Y$. A {\em weak isomorphism of pairs} $\Phi\colon (G,\famS) \to (H, \cal T)$ is an isomorphism $\Phi\colon G\to H$ such that for some bijection $f\colon X \to Y$ each group $\Phi(S_x)$ is a conjugate in $H$ of $T_{f(x)}$.
\end{defn}
\begin{defn}
If $\Sigma$ is a collection of subgroups of the discrete group $\Gamma$ define the {\em profinite completion} of $(\Gamma, \Sigma)$ to be the group pair $(G, \famS)$ where $G= \widehat \Gamma$ and $\famS$ consists of the closure in $G$ of each group in $\Sigma$. 
\end{defn}
Note that the groups in $\famS$ may not be equal the profinite completions of the groups in $\Sigma$ in the absence of conditions on the profinite topology of $\Gamma$---specifically the condition that $\Gamma$ induces the full profinite topology on each element of $\Sigma$ in the following sense.
\begin{defn}
Let $\Gamma$ be a discrete group and let $\Lambda\leq \Gamma$. Then we say that $\Lambda$ is {\em separable} in $\Gamma$ if for every $g\in\Gamma\smallsetminus \Lambda$ there is a map $\phi$ from $\Gamma$ to a finite group such that $\phi(g)\notin \phi(\Lambda)$.

We say that {\em $\Gamma$ induces the full profinite topology on $\Lambda$} if for every finite index normal subgroup $U$ of $\Lambda$ there is a finite index normal subgroup $V$ of $\Gamma$ and with $V\cap \Lambda\leq U$.

We say that $\Lambda$ is {\em fully separable} in $\Gamma$ if it is separable in $\Gamma$ and $\Gamma$ induces the full profinite topology on $\Lambda$.
\end{defn}
In our case we shall only be considering families of subgroups resulting from incompressible boundary components of 3-manifolds. There are no problems with the profinite topology on these subgroups because of the following theorem, which builds on work of Przytycki and Wise \cite{PW14} among others.
\begin{theorem}[Corollary 6.20 of \cite{WilkesRelCoh}]\label{FullSepBdy}
Let $M$ be a compact 3-manifold with $\pi_1$-injective boundary and let $L$ be a boundary component of $M$. Then $\pi_1 L$ is fully separable in $\pi_1 M$. 
\end{theorem}
The primary tool we shall use to study profinite group pairs is their relative cohomology. There is no need for us to define this theory here as we shall only need properties as black boxes. Suffice it to say that for profinite group pairs there is a theory of relative (co)homology which has the aspects one might expect, viz:
\begin{itemize}
\item a long exact sequence of relative cohomology \cite[Proposition 2.4]{WilkesRelCoh};
\item functoriality with respect to sensible maps of group pairs \cite[Proposition 2.6]{WilkesRelCoh};
\item invariance under replacing subgroups by conjugates \cite[Proposition 2.9]{WilkesRelCoh};
\item a theory of cohomological dimension $\cd_p(G,\famS)$ with respect to a prime $p$ \cite[Section 2.3]{WilkesRelCoh};
\item a notion of cup product \cite[Section 3]{WilkesRelCoh}; and
\item a notion of Poincar\'e duality (or PD$^n$) pair (with respect to a set of prime numbers) \cite[Section 5]{WilkesRelCoh}.
\end{itemize} 
The relevant results we shall primarily be using are the following.
\begin{theorem}[Theorem 6.21 of \cite{WilkesRelCoh}]\label{ProfPD3}
Let $M$ be a compact aspherical 3-manifold with incompressible boundary components $\bdy M_1,\ldots \bdy M_r$. Let $\Gamma = \pi_1 M$ and let $\Sigma= \{\pi_1 \bdy M_i\}_{1\leq i\leq r}$. Then the profinite completion of $(\Gamma, \Sigma)$ is a PD$^3$ pair at every prime $p$.
\end{theorem}
\begin{theorem}[Corollary 5.14 of \cite{WilkesRelCoh}]\label{FixedVertex}
Let $(G,\famS)$ be a profinite group pair which is a PD$^n$ pair at every prime $p$. Suppose that $G$ acts on a profinite tree $T$. Suppose that for every edge $e$ of $T$ we have $\cd_p(G_e, \famS^{G_e})<n-1$ for all $p$ where $G_e$ denotes the stabiliser of $e$. Then $G$ fixes a vertex of $T$.
\end{theorem}
\begin{lem}[Lemma 2.20 of \cite{WilkesRelCoh}]\label{LemReductionOfFamily}
Let $(G,\famS)$ be a profinite group pair. Suppose that the family $\famS$ is such that at most one subgroup $S_0$ is non-trivial. Then for every $k> 1$ and every discrete torsion $G$-module $A$, we have 
\[H^k(G,\famS; A)=H^k(G, \{S_0\};A)  \]
In particular, $\cd_p(G,\famS)=\cd_p(G,\{S_0\})$ for all $p\in \pi$ except possibly if one  dimension is 1 and the other 0.
\end{lem}

%% file: JSJbackgd.tex
\subsection{JSJ decompositions of 3-manifolds}\label{subsecJSJbackgd}
\begin{defn}
A 3-manifold $M$ is {\em atoroidal} if any embedded incompressible torus is isotopic to a boundary component of $M$. Similarly $M$ is {\em \anan} if any properly embedded incompressible annulus is isotopic into a boundary component of $M$.
\end{defn}

The JSJ decomposition of a closed irreducible 3-manifold \cite{JS78, Joh79} consists of a canonical collection of disjoint incompressible tori embedded in the manifold such that, on removing small open neighbourhoods of these tori, the connected connected components of the remainder are either Seifert fibred or atoroidal.

In this paper we shall be considering the analogous decomposition for a compact irreducible 3-manifold $M$ with incompressible boundary. Here the decomposition consists of cutting along both annuli and tori embedded in the manifold. Our source for this decomposition is \cite{NS97}. The classification of the pieces of this decomposition is not quite so clean as in the closed case.  
\begin{defn}
Let $(X, M_\bullet)$ be a graph-of-spaces decomposition of a 3-manifold $M$ whose edge spaces are annuli and tori. For a vertex space $M_x$ let $\bdy_0 M_x$ be the part of $\bdy M_x$ coming from edge spaces $M_e$ and let $\bdy_1 M_x$ be the portion of $\bdy M_x$ coming from $\bdy M$. Then $M_x$ is {\em simple} if any essential annulus $(A, \bdy A)\subseteq (M_x, \bdy_0 M_x)$ is parallel to $\bdy_1 M_x$.
\end{defn}
\begin{theorem}[see Section 3 of \cite{NS97}]
	Let $M$ be a compact irreducible 3-manifold with incompressible boundary. There exists a minimal collection of essential disjoint annuli and tori (the {\em JSJ} annuli and tori) properly embedded in $M$ such that the complement of a regular neighbourhood of the union of these surfaces consists of simple atoroidal manifolds, Seifert fibred manifolds, and I-bundles. This collection is unique up to isotopy. 
\end{theorem}
We refer to the graph of spaces so obtained, whose edge spaces are the annuli and tori and whose vertex spaces are closures in $M$ of components of the complement of a regular neighbourhood of the union of these surfaces, as the {\em JSJ decomposition} of $M$. 

Note that the simple atoroidal pieces may not be an-annular: they may contain essential annuli whose boundaries run over the JSJ annuli of $M$. 

Let $D'M$ be the manifold obtained from two copies of $M$ by identifying the copies of each boundary component of $M$ which is not a torus. We say we have `doubled $M$ along its higher-genus boundary'. We use the symbol $D'M$ rather than $DM$ to remind the reader that we have not necessarily doubled along all boundary components of $M$, but only the higher-genus ones. Then $D'M$ has toroidal boundary and has a JSJ decomposition along only tori. These tori are either copies of the JSJ tori of $M$ or are the doubles of JSJ annuli of $M$ along their boundary curves. Notice that the obvious folding map $D'M\to M$ carries JSJ pieces of $D'M$ to JSJ pieces of $M$.

The JSJ decomposition of a compact irreducible 3-manifold $M$ with incompressible boundary induces a graph-of-groups decomposition of its fundamental group whose edge groups are abelian and whose vertex groups are the fundamental groups of the corresponding vertex spaces. 

By the comments above, the obvious retraction $\rho\colon \pi_1 (D'M) \to \pi_1 (M)$ induces a map of graphs of groups. More precisely let the graph of groups corresponding to the JSJ decomposition of $D'M$ be $(X',\Gamma'_\bullet)$. Let $\Z/2 = \gp{\tau}$ act on $D'M$ by swapping the two copies of $M$. This action descends to an action on $\pi_1 (D'M)$ and, by uniqueness of JSJ decompositions, to an action of $\tau$ on $X'$ such that $\Gamma'_{\tau\cdot x} = \tau\cdot\Gamma'_x$ for all $x\in X'$. Then the graph of groups for the JSJ decomposition of $M$ is $(X'/\gp{\tau}, \rho(\Gamma'_\bullet))$ where the group corresponding to a point $\{x, \tau\cdot x\}$ of $X'/\gp{\tau}$ is $\rho(\Gamma'_x) = \rho(\tau\cdot\Gamma'_x)$. Note that there is also a section $\iota\colon X\to X'$ of this quotient induced by the inclusion of $M$ in $D'M$, and we have $\Gamma_x\subseteq\Gamma'_{\iota(x)}$.

%% file: ProfProps.tex
\section{Profinite properties of JSJ decompositions of compact 3-manifolds}
Several useful properties of the JSJ decomposition of a closed 3-manifold were proved by Wilton and Zalesskii \cite[Theorems A and B]{WZ10}. In this section we note that these results extend to the bounded case. First we require a notion of when a graph of discrete groups is well-behaved.

We will denote the fundamental group of a graph of discrete groups $(X,\Gamma_\bullet)$ by $\pi_1(X,\Gamma_\bullet)$, and the fundamental group of a graph of profinite groups $(X,G_\bullet)$ by $\Pi_1(X,G_\bullet)$. See Section 6.2 of \cite{Ribes17} for information on graphs of profinite groups.
\begin{defn}
A graph of discrete groups $(X,\Gamma_\bullet)$ is {\em efficient} if $\pi_1(X,\Gamma_\bullet)$ is residually finite, each group $\Gamma_x$ is closed in the profinite topology on $\pi_1(X,\Gamma_\bullet)$, and $\pi_1(X, \Gamma_\bullet)$ induces the full profinite topology on each $\Gamma_x$.
\end{defn}
\begin{theorem}[Exercise 9.2.7 of \cite{RZ00}]\label{EffGivesInjective}
Let $(X,\Gamma_\bullet)$ be an efficient finite graph of discrete groups. Then $(X,\widehat \Gamma_\bullet)$ is an injective graph of profinite groups and \[\widehat{\pi_1(X, \Gamma_\bullet)}\iso\Pi_1(X, \widehat{\Gamma}_\bullet)\]
\end{theorem}
Here an `injective' graph of groups is one for which the canonical $\widehat{\Gamma}_\bullet\to\Pi_1(X, \widehat{\Gamma}_\bullet)$ are inclusions. This is automatic for graphs of discrete groups, but not for profinite groups.

Let $G=\widehat{\pi_1 M}$ and let $\famS$ be the collection of closed subgroups of $G$ consisting of the closures in $G$ of the fundamental groups of boundary components of $M$. By Theorem \ref{FullSepBdy} these closures are precisely the profinite completions of the respective fundamental groups of boundary components. It easily follows that doubling along them will give efficient graphs of groups. 

\begin{prop}
Let $M$ be a compact irreducible 3-manifold with incompressible boundary. Let $\Gamma=\pi_1 M$, let the JSJ decomposition of $M$ be $(X, M_\bullet)$ and let $\Gamma_\bullet=\pi_1 M_\bullet$. Then the graph of groups $(X, \Gamma_\bullet)$ is efficient. 
\end{prop}
\begin{proof}
For the case of manifolds with toroidal or empty boundary this is Theorem A of \cite{WZ10}. We will deduce the general case from this. Let $D'M$ be the double of $M$ along its higher-genus boundary and let the fundamental group of $D'M$ be $\Gamma'$. Let the graph of groups corresponding to the JSJ decomposition of $D'M$ be $(X',\Gamma'_\bullet)$. This graph of groups is efficient by the toroidal boundary case. Let $\rho\colon \Gamma' \to \Gamma$ be the retraction and let $\iota\colon X\to X'$ be the section defined in Section \ref{subsecJSJbackgd}.

Let $x\in X$ and let $x'=\iota(x)$. If $g\in \Gamma\smallsetminus \Gamma_x$ then $g\in \Gamma'\smallsetminus \Gamma'_{x'}$, so there is a finite quotient of $\Gamma'$ distinguishing $g$ from $\Gamma'_{x'}$, hence from $\Gamma_x$. So $\Gamma_x$ is separable in $\Gamma$.

Let $U\nsgp[f]\Gamma_x$ be a finite index normal subgroup of $\Gamma_x$. Then $\rho^{-1}(U)\cap \Gamma'_{x'}$ is a finite index normal subgroup of $\Gamma'_{x'}$ which intersects $\Gamma_x$ in precisely $U$. There is a finite index subgroup $V$ of $\Gamma'$ such that $V\cap \Gamma'_{x'}$ is contained in $\rho^{-1}(U)\cap \Gamma'_{x'}$. Then $V\cap \Gamma$ is a finite index subgroup of $\Gamma$ whose intersection with $\Gamma_x$ is contained in $U$. So $\Gamma$ induces the full profinite topology on $\Gamma_x$ and we are done.  
\end{proof}
Retaining the notation of the previous proposition, we now have two injective graphs of profinite groups $(X,\widehat\Gamma_\bullet)$ and $(X',\widehat\Gamma'_\bullet)$. Because $\Gamma'$ induces the full profinite topology on its retract $\Gamma$, the section $\iota$ still induces inclusions $\widehat\Gamma_x\to \widehat{\Gamma}'_{\iota(x)}$. Furthermore $\Z/2$ still acts on doubles in the same way as in Section \ref{subsecJSJbackgd}. The efficiency of doubling operations and of JSJ decompositions implies that the analysis at the end of Section \ref{subsecJSJbackgd} still applies. We collect this as a proposition for reference later.
\begin{prop}\label{retractionofJSJ}
Let $M$ be a compact irreducible 3-manifold with incompressible boundary. Let $\Gamma=\pi_1 M$, let the JSJ decomposition of $M$ be $(X, M_\bullet)$ and let $\Gamma_\bullet=\pi_1 M_\bullet$. Further let $D'M$ be the double of $M$ along its higher genus boundary and let $(X, \Gamma'_\bullet)$ be the graph of groups decomposition of $\Gamma' = \pi_1D'M$ corresponding to the JSJ decomposition of $D'M$.

Let $\Z/2 = \gp{\tau}$ act on $D'M$ by swapping the two copies of $M$. Then this action descends to an action on $\widehat\Gamma'$ which commutes with the natural retraction $\rho\colon\widehat\Gamma'\to \widehat\Gamma$. Also $\widehat\Gamma'_{\tau\cdot x} = \tau\cdot\widehat\Gamma'_x$ for all $x\in X'$ and there is an equality of graphs of groups
\[(X,\widehat\Gamma_\bullet) = (X'/\gp{\tau}, \rho(\widehat\Gamma'_\bullet)) \] 
where the group corresponding to a point $\{x, \tau\cdot x\}$ of $X'/\gp{\tau}$ is $\rho(\widehat\Gamma'_x) = \rho(\tau\cdot\widehat\Gamma'_x)$. Finally there is also section $\iota\colon X\to X'$ of this quotient induced by the inclusion of $M$ in $D'M$, and $\widehat\Gamma_x\subseteq\widehat\Gamma'_{\iota(x)}$.
\end{prop} 
Finally we make an observation that the graph of groups $(X,\widehat\Gamma_\bullet)$ has a useful property called {\em acylindricity}. Just as in classical Bass-Serre theory these injective graphs of groups gives actions of the fundamental groups $G=\widehat{\Gamma}$ and $G'=\widehat{\Gamma'}$ on profinite trees, called the {\em standard graphs} of the graphs of groups. Denote these standard graphs by $T$ and $T'$. See \cite[Section 2.4]{Ribes17} for the definition of a profinite tree and \cite[Section 6.3]{Ribes17} for the definition of the standard graph. By construction of the standard graph there is a $G$-equivariant inclusion $j\colon T\to T'$. 

Recall that an action of a profinite group $G$ on a profinite tree $T$ is {\em $k$-acylindrical} if the stabiliser of any path of length greater than $k$ is trivial. If a path in $T$ has non-trivial stabiliser then the image of this path under $j$ is a path of the same length in $T'$ with non-trivial stabiliser. Since $G'$ is the profinite completion of a 3-manifold with toroidal boundary and its action on $T'$ is that coming from the JSJ decomposition, the action of $G'$ on $T'$ is $k$-acylindrical (\cite[Proposition 6.8]{Wilkes16} or \cite[Lemma 4.11]{HWZ12} and \cite[Lemma 4.5]{WZ14}) with $k$ equal to 1, 2 or 4 depending on the manifold. We therefore have the following proposition. 
\begin{prop}
Let $M$ be a compact irreducible 3-manifold with incompressible boundary. Let $\Gamma=\pi_1 M$, let the JSJ decomposition of $M$ be $(X, M_\bullet)$ and let $\Gamma_\bullet=\pi_1 M_\bullet$. Then the action of $\widehat\Gamma$ on the standard graph of the graph of groups $(X, \widehat\Gamma_\bullet)$ is acylindrical. 

More precisely let $D'M$ be the double of $M$ along its higher genus boundary and let $(X, \Gamma'_\bullet)$ be the graph of groups decomposition of $\Gamma' = \pi_1(D'M)$ corresponding to the JSJ decomposition of $D'M$. If the action of $\widehat\Gamma'$ on the standard graph of the graph of groups $(X', \widehat\Gamma'_\bullet)$ is $k$-acylindrical then the standard graph of the graph of groups $(X, \widehat\Gamma_\bullet)$ is $k'$-acylindrical for some $k'\leq k$.
\end{prop}
\begin{rmk}
The inequality in the final part of the theorem may be strict in the following case. If the JSJ pieces of $M$ consist of \anan\ atoroidal manifolds and at least one $I$-bundle then the action of $\widehat{\Gamma}$ is 1-acylindrical; however $D'M$ has a Seifert fibred piece and the action of $\widehat{\Gamma}'$ is 2-acylindrical.
\end{rmk}

%% file: ProfAtorAnan.tex
\section{Atoroidality and an-annularity for profinite group pairs}
Throughout this section let $(G,\famS)$ be a profinite group pair.
\begin{defn}
Let $(G,\famS)$ be a profinite group pair. We define the following notions.
\begin{itemize}
\item $(G,\famS)$ is {\em atoroidal} if every abelian subgroup of $G$ is either procyclic or conjugate into an abelian member of $\famS$. If in addition the collection of abelian members of \famS\ is malnormal then we say the group pair is {\em strictly atoroidal}.
\item $(G,\famS)$ is \anan\ if for every procyclic subgroup $A\leq G$ we have $\|\famS^A\| \leq 1$.
\end{itemize}
\end{defn}
The \anan{}ity property is in fact equivalent to malnormality of $\famS$ in $G$ (Propsition \ref{MalnImpliesAnan} below), but provides a better formulation for the application of relative cohomology theory.

The definition of an-annular states roughly that each cyclic subgroup has at most one intersection with a peripheral subgroup. However {\it a priori} such a single intersection could still represent something like a M\"obius band properly embedded in a manifold. However this, and any similar situations that could arise in a profinite group, are ruled out by the following lemma (which may perhaps be thought of as a group theoretic analogue of taking a regular neighbourhood of a M\"obius band to obtain an annulus). 
\begin{lem}\label{nomobiusbands}
Suppose $(G,\famS)$ is \anan. Let $A\leq G$ be a procyclic subgroup with $\|\famS^A\| = 1$. Then the unique non-trivial member of $\famS^A$ is equal to $A$. 
\end{lem}
\begin{proof}
Let $i\in I$ and $g\in G$ be such that $A\cap gS_i g^{-1} = A'\neq 1$. We will show that $\|\famS^{A'}\|\geq 2$, contradicting the \anan{}ity condition. For let $x\in A\smallsetminus A'$ and note that $a^{-1}x\notin gS_i g^{-1}$ for all $a\in A'$. This implies $A'gS_i\neq A'xgS_i$. Also, since $A$ is abelian we have 
\[A' = xA'x^{-1} = A\cap xg S_i (xg)^{-1} \neq 1 \neq A'\cap g S_i g^{-1} = A'\]
Thus the distinct elements $A'gS_i$ and $A'xgS_i$ in the indexing set $A'\lqt G/\famS$ both give non-trivial elements of $\famS^{A'}$. Hence $\|\famS^{A'}\|\geq 2$ as required.
\end{proof}
\begin{prop}\label{MalnImpliesAnan}
The family $\famS$ is a malnormal collection in $G$ if and only if $(G,\famS)$ is \anan.
\end{prop}
\begin{proof}
Assume \famS\ is malnormal. Let $A\leq G$ be a procyclic subgroup and suppose $\|\famS^A\|\geq 2$. Then there exist $i,j\in I$ and $g,h\in G$ such that either $i\neq j$ or $i=j$ and $AgS_i\neq AhS_i$, and such that 
\[A\cap gS_ig^{-1} \neq 1\neq A\cap hS_jh^{-1} \]
Firstly note that both of these intersections equal $A$. For if, say, there exists $x\in A\smallsetminus gS_ig^{-1}$ then since $A$ is abelian we have 
\[A\cap gS_ig^{-1} = A\cap xgS_i(xg)^{-1} \]
whence $S_i\cap g^{-1}xgS_i(g^{-1}xg)^{-1}\neq 1$. Then by malnormality we would have $g^{-1}xg\in S_i$, a contradiction.

Therefore we have $A\leq gS_ig^{-1} \cap hS_jh^{-1}$ whence malnormality implies $i=j$ and $g^{-1}h\in S_i$ so that $AgS_i=AhS_j$, again a contradiction.

For the converse, if \famS\ is not malnormal then there is a non-trivial cyclic subgroup $1\neq A\subseteq S_i\cap S_j^g$ where either $i\neq j$ or $i=j$ and $g\notin S_j$. Then $A1S_i$ and $Ag^{-1}S_j$ are distinct elements of $A\lqt G/\famS$ with 
\[ A\cap 1S_i1^{-1} \neq 1\neq A\cap g^{-1}S_jg \]
hence $\|\famS^A\|\geq 2$ and we are done.
\end{proof}
\begin{theorem}\label{AnanFixVertex}
Let $(G,\famS)$ be a profinite group pair which is a PD$^n$ pair at every prime $p$, for some $n\geq 3$. Suppose $(G,\famS)$ is strictly atoroidal and \anan. Then any action of $G$ on a profinite tree $T$ with abelian edge stabilisers fixes a vertex of $T$.
\end{theorem} 
\begin{proof}
If not, then by Theorem \ref{FixedVertex} there exists a prime $p$ and an edge $e$ of $T$ such that $\cd_p(G_e, \famS^{G_e})\geq 2$, where $G_e$ denotes the stabiliser of $e$. We show that this is not the case. Let $e$ be an edge of $T$.

If $G_e$ is not cyclic then by strict atoroidality $G_e$ intersects a unique conjugate of a member of $\famS$ and is contained within it. Therefore the collection $\famS^{G_e}$ contains exactly one non-trivial member which is $G_e$ itself.  Therefore by Lemma \ref{LemReductionOfFamily} we have
\[H^k(G_e, \famS^{G_e}) = H^k(G_e,\{G_e\}) = 0 \text{ for all $k\geq 2$} \]
with coefficients in any $G_e$-module. Hence $\cd_p(G_e, \famS^{G_e})<2$.

Otherwise if $G_e$ is cyclic then by hypothesis $\|\famS^{G_e}\|\leq 1$. If $\|\famS^{G_e}\| = 0$ then $\cd_p(G_e, \famS^{G_e})=1$. If $\|\famS^{G_e}\| = 1$ then by Lemma \ref{nomobiusbands} the one non-trivial member of $\famS^{G_e}$ is $G_e$ itself and, as before, $\cd_p(G_e, \famS^{G_e})\leq 1$.
\end{proof}

%% file: MalnPrComp.tex
\section{Malnormality in the profinite completion} 
We first set up some notation for the section. Let $M$ be a compact 3-manifold with non-empty incompressible boundary. Denote by $\pi_1\bdy M$ a collection of subgroups of $\pi_1 M$ containing one conjugacy representative of the fundamental group of each boundary component of $M$, indexed over some finite set $I$. Let $(G,\famS)$ be the profinite completion of the pair $(\pi_1 M, \pi_1\bdy M)$.

The following malnormality result is closely related to the results of \cite{WZ14} but does not appear in the precise form we require. We will therefore deduce it from those previous results by a doubling argument.
\begin{prop}\label{MalnBdy}
Suppose $M$ is irreducible, atoroidal and \anan. Then \famS\ is malnormal in $G$.
\end{prop} 
\begin{proof}
Let $N$ be the 3-manifold obtained by doubling $M$ along those boundary components which are not tori. By a standard topological argument the atoroidality and \anan{}ity of $M$ imply that $N$ is atoroidal. Then $N$ is an atoroidal irreducible Haken 3-manifold (and is not the orientable $I$-bundle over a Klein bottle) and is therefore cusped-hyperbolic by Thurston's hyperbolisation theorem \cite{Thurston86}. Furthermore $M$ (or rather its interior) is infinite-volume hyperbolic by the same theorem. Furthermore the decomposition of $\pi_1 N$ along the former boundary components of $M$ is efficient, as may be proved using the retraction $\pi_1 N\to \pi_1 M$. 

We shall denote by $\widetilde M$ the second copy of $M$ contained in $N$ and use tildes to denote the canonical isomorphism from $\pi_1 M$ to $\pi_1 \widetilde M$. Similarly for their profinite completions, boundary components et cetera. The graph of groups decomposition of $N$ consists of two vertices and several edges between them. Choose one edge to give a maximal subtree and let $t_j$ denote the stable letter for the HNN extension along the edge $e$ corresponding to the boundary component indexed by $j$ (with the understanding that the stable letter for the chosen edge $e$ is the identity). Hence in $H=\widehat{\pi_1 N}$ conjugation by $t_i$ gives the standard isomorphism $S_i\to\widetilde{S}_i$ for the non-abelian $S_i$.

Let $i,j\in I$ and suppose that $S_i\cap S_j^g\neq 1$ for some $g\in G$. We must show that $i=j$ and $g\in S_i$. We now break into cases depending on whether $S_i$ and $S_j$ come from toroidal boundary components of $M$ or not.
\begin{enumerate}[{Case} 1]
\item If $S_i$ and $S_j$ are abelian, then they persist as peripheral subgroups of $\widehat{\pi_1 N}$. Lemma 4.5 of \cite{WZ14} informs us that the profinite completions of the fundamental groups of the remaining toroidal boundary components of $N$ form a malnormal collection in $\widehat{\pi_1 N}$, so we are done in this case.
\item Suppose $S_j$ is abelian and $S_i$ is not, and let $S_i\cap S_j^g= A\neq 1$. By symmetry we have $\widetilde S_i \cap \widetilde S_j^{\tilde g} = \widetilde A$. Hence $S_j^{t_i}\cap \widetilde{S}_j = \widetilde A\neq 1$, which is impossible as $S_j$ and $\widetilde S_j$ are distinct peripheral subgroups of $H$ (again using Lemma 4.5 of \cite{WZ14}).
\item Suppose $i=j$ and $S_i$ is non-abelian. Since $\pi_1 M$ is virtually compact special (Theorem 14.29 and paragraph before Corollary 14.33 of \cite{Wise11}) and is hyperbolic relative to its toral peripheral subgroups, and since $M$ is \anan\ so that $\pi_1 \bdy M$ is a malnormal collection we may apply Theorem 4.2 of \cite{WZ14} to conclude that $S_i$ is malnormal relative to the abelian groups in \famS. That is, if $S_i\cap S_i^g$ is not conjugate into an abelian group of \famS\ then $g\in S_i$. However by Case 2 this intersection meets any conjugate of an abelian group trivially so $S_i$ is in fact absolutely malnormal.
\item Suppose $i\neq j$ and that $S_i$ and $S_j$ are non-abelian. Without loss of generality $t_j\neq 1$. If $S_i\cap S_j^g= A\neq 1$ then also $\widetilde S_i\cap \widetilde S_j^{\tilde g}= \widetilde A\neq 1$. It easily follows that $S_j\cap S_j{}^\wedge \{t_j\tilde g t_i^{-1} g^{-1}\}\neq 1$. Since $t_j\tilde g t_i^{-1} g^{-1}\notin S_j$ (consider the homomorphism to $\widehat\Z$ corresponding to the stable letter $t_j$) this contradicts Case 3 and we are done.\qedhere
\end{enumerate}
\end{proof}
Combining with Proposition \ref{MalnImpliesAnan} gives the following corollary.
\begin{clly}\label{ProfAnan}
Suppose $M$ is irreducible, atoroidal and \anan. Then $(G,\famS)$ is strictly atoroidal and \anan.
\end{clly}
\begin{proof}
We have now proven everything except the atoroidality condition; this follows from Lemma 4.5 of \cite{WZ14} via the same doubling as in Proposition \ref{MalnBdy}.
\end{proof}

%% file: MainTheorems.tex
\section{Main theorems}
\begin{theorem}\label{MainAnanThm}
Let $M$ be a compact irreducible orientable 3-manifold with incompressible boundary. Assume that $M$ is atoroidal and \anan. If $G=\widehat{\pi_1 M}$ acts on a profinite tree with abelian edge stabilisers then $G$ fixes a unique vertex. 
\end{theorem}
\begin{proof}
This follows from Theorem \ref{AnanFixVertex} and Corollary \ref{ProfAnan} together with Theorem \ref{ProfPD3}.
\end{proof}
\begin{rmk}
This is our replacement and improvement for Lemmas 3.4 and 4.4 of \cite{WZ17}. It is an improvement in two senses: it extends the result to more general boundaries than toral, and removes the assumption that the action is acylindrical.
\end{rmk}
The following definition will be useful for our purposes.
\begin{defn}
Let $(X, G_\bullet)$ and $(Y,H_\bullet)$ be graphs of profinite groups with fundamental groups $G$ and $H$ respectively. A {\em preservation of decompositions} $(f,\Phi)\colon (X, G_\bullet)\to(Y,H_\bullet)$ is a pair of maps where $f\colon X \to Y$ is a graph isomorphism and $\Phi\colon G\to H$ is an isomorphism of profinite groups such that $\Phi(G_x)$ is a conjugate of $H_{f(x)}$ for all $x\in X$.
\end{defn}
\begin{theorem}[= Theorem B of \cite{WZ17}]\label{JSJtorbdy}
Let $M$ and $N$ be irreducible orientable 3-manifolds with toroidal boundary and let their respective JSJ decompositions be $(X,M_\bullet)$ and $(Y,N_\bullet)$. Suppose there exists a weak isomorphism of pairs $\Phi\colon \widehat{\pi_1 M}\to \widehat{\pi_1 N}$. Then there exists a graph isomorphism $f\colon X\to Y$ such that \[(f,\Phi)\colon (X,\widehat{\pi_1 M_x}) \to (Y,\widehat{\pi_1 N_x})\] is a preservation of decompositions.
\end{theorem}
\begin{proof}
As this theorem is already known we shall only sketch the proof. First recall that in the case of toroidal boundary one only has tori in the JSJ decomposition and therefore all pieces are Seifert fibred or cusped hyperbolic. One may either follow the proof as given in Section 4 of \cite{WZ17} and substitute Theorem \ref{MainAnanThm} to handle hyperbolic pieces; or one may use Theorem 6.2 of \cite{Wilkes16} to detect the Seifert-fibred portions of the JSJ graph, apply Theorem \ref{MainAnanThm} to detect hyperbolic pieces and then locate the remaining edge groups via the intersections of these hyperbolic pieces with each other and with Seifert fibred pieces. 
\end{proof}

\begin{theorem}\label{JSJgeneral}
Let $M$ and $N$ be compact irreducible orientable 3-manifolds with incompressible boundary and let their respective JSJ decompositions be $(X,M_\bullet)$ and $(Y,N_\bullet)$. Let $(\pi_1\bdy M)_{\rm hg}$ and $(\pi_1\bdy N)_{\rm hg}$ denote the families of peripheral subgroups corresponding to (one conjugate of) the fundamental group of each higher-genus boundary component, and let $\famS_{\rm hg}$ and ${\cal T}_{\rm hg}$ be their profinite completions. Suppose there exists an isomorphism of group pairs $\Phi\colon (\widehat{\pi_1 M}, \famS_{\rm hg}) \to (\widehat{\pi_1 N}, {\cal T}_{\rm hg})$. Then there exists a graph isomorphism $f\colon X\to Y$ such that 
\[(f,\Phi)\colon (X,\widehat{\pi_1 M_x}) \to (Y,\widehat{\pi_1 N_x})\]
 is a preservation of decompositions.
\end{theorem}
\begin{proof}
Let $G= \widehat{\pi_1 M}$ and $H=\widehat{\pi_1 N}$. Let $\famS_{\rm hg}=\{P_1,\ldots, P_n\}$ and ${\cal T}_{\rm hg}=\{Q_1,\ldots, Q_n\}$, indexed so that $\Phi(P_i) = Q_i^{g_i}$ for some $g_i\in G$.

Take copies $\widetilde G$ and $\widetilde H$ of $G$ and $H$. We will use tildes $\widetilde{(\,\cdot\,)}$ to denote the translation of an element or map on $G$ or $H$ to the copy $\widetilde G$ or $\widetilde H$. Form the high-genus double $D'G$ of $G$ as the fundamental group of the natural graph of groups with vertex groups $G$ and $\widetilde G$ with respect to a maximal subtree consisting of the edge with edge group $P_1$. Let the stable letter for another edge group $P_i$ be $s_i$. That is, in $D'G$ we have $\widetilde P_i = P_i^{s_i}$. Similarly from $D'H$ with respect to the subtree with edge labelled by $Q_1$ and let the stable letter for an edge group $Q_i$ be $t_i$. 

The map $\Phi$ induces an isomorphism $\Psi=D'\Phi\colon D'G \to D'H$ defined by 
\[\Psi(g) = \Phi(g)\text{ for $g\in G$}, \quad \Psi(\tilde g) = \widetilde\Phi(\tilde g)^{g_1}\text{ for $\tilde g\in \widetilde G$},\quad s_i \mapsto g_i^{-1} t_i \tilde g_i  \]
The reader may readily check that this is a well-defined isomorphism of graphs of groups. Note that if $\rho_G$ denotes the canonical retraction $D'G\to G$ and $\rho_H$ denotes the canonical retraction $D'H\to H$ then there is a commuting diagram
\[\begin{tikzcd}
D'G \ar{r}{\Psi}[swap]{\iso} \ar{d}{\rho_G}& D'H \ar{d}{\rho_H} \\
G \ar{r}{\Phi}[swap]{\iso} & H
\end{tikzcd}\]  
Furthermore if $\tau$ denotes the action of $\Z/2$ on a double by swapping the two copies of $G$ (or $H$) then $\tau$ commutes with $\Psi$. More precisely $\tau$ is the map which swaps the two copies $g$ and $\tilde g$ of any element of $G$ and sends each $s_i$ to its inverse, and similarly for $H$.

Let the graph of groups decompositions of $D'G$ and $D'H$ corresponding to the JSJ decompositions of $D'M$ and $D'N$ be $(X',G'_\bullet)$ and $(Y',H'_\bullet)$. By Theorem \ref{JSJtorbdy} applied both to $\Psi$ and to $\tau$ there is a commuting diagram of preservations of decompositions 
\[\begin{tikzcd}
(X',  G'_\bullet)\ar{r}{(f, \Psi)}\ar{d}{(\tau, \tau)} & (Y',  H'_\bullet)\ar{d}{(\tau, \tau)}\\
(X',  G'_\bullet)\ar{r}{(f, \Psi)} & (Y',  H'_\bullet)
\end{tikzcd}\]
By Proposition \ref{retractionofJSJ} the graphs of groups decomposition of $G$ given by the JSJ decompositions of $M$ is $(X'/\gp{\tau}, \rho_G(G'_\bullet))$. Similarly the decomposition of $H$ is $(Y'/\gp{\tau}, \rho_H(H'_\bullet))$. The two commutative diagrams above now imply that there is a preservation of decompositions $(f/\gp{\tau}, \Phi)$ from one JSJ decomposition to the other. This concludes the proof.
\end{proof}
\begin{rmk}
The reason that a doubling argument (and hence some constraint upon the isomorphisms with respect to peripheral structure) seems necessary is that the JSJ decompositions may contain $I$-bundles over surfaces-with-boundary. These pieces have free fundamental groups and therefore their actions on (profinite) trees could be quite wild. In particular there is no {\it a priori} reason for them to fix a vertex, so the proof strategy of Theorem \ref{JSJtorbdy} (showing that each JSJ piece fixes a vertex) breaks down.
\end{rmk}